\documentclass[leqno,12pt]{amsart} 

\usepackage{amsmath,amssymb,verbatim}

\usepackage[pagebackref,pdftex,hyperindex]{hyperref}

\usepackage{color}
\usepackage{braket}
\usepackage{bm}
\usepackage{url}
\usepackage{comment}
\usepackage[all,cmtip]{xy}

\theoremstyle{plain} 
\newtheorem{theorem}{\noindent\bf Theorem}[section] 
\newtheorem{lemma}[theorem]{\noindent\bf Lemma}

\newtheorem{proposition}[theorem]{\noindent\bf Proposition}

\theoremstyle{definition} 
\newtheorem{definition}[theorem]{\noindent\bf Definition}
\newtheorem{remark}[theorem]{\noindent\it Remark}

\newcommand{\T}[0]{\mathbb{T}}
\newcommand{\C}[0]{\mathbb{C}}
\newcommand{\R}[0]{\mathbb{R}}
\newcommand{\SympA}[0]{\operatorname{Symp}(\T^2,\sigma)}
\newcommand{\SympB}[0]{\operatorname{Symp}(\T^4,\omega)}

\newcommand{\LD}[2]{\langle\langle\,#1\,,\,#2\,\rangle\rangle}
\newcommand{\pdt}[0]{\left.\frac{\partial}{\partial t}\right|_{t=0}}

\numberwithin{equation}{section}

\theoremstyle{remark}

\begin{document}

\title[Symplectic gradient flow on the Hyperkähler torus]
{Local existence of the symplectic gradient flow on the HyperKähler four-dimensional flat torus} 
\author[L. Pinsard Morel]
{Lucas Pinsard Morel}

\thanks{
}
\address{
Lucas Pinsard Morel:
Faculté des Sciences et des Techniques, 2 Chem. de la Houssinière, 44322 Nantes,
France
}
\email{lucas.morel1@univ-nantes.fr}

\begin{abstract}
Introducing a moment map whose zero locus is the group of symplectomorphisms of the real four-dimensional torus, we exhibit a gradient flow that can be made into a strictly parabolic flow by mean of a DeTurck trick (famously known for its use in the study of the Ricci flow), showing the local existence and regularity for the solutions of this flow and hence showing that the group of symplectomorphisms of the real four-dimensional torus is locally contractible. This work follows the ideas introduced by Yann Rollin in \cite{YR24}, even though the moment map picture comes from different considerations. 
\end{abstract}

\maketitle

\tableofcontents

\section{Introduction}
\subsection{Motivations}
The group of symplectomorphisms of a symplectic manifold is an important object in symplectic geometric, but very few things are known about its topology \cite[\S$10.4$]{MS16}. As an infinite dimensional Lie group, its local structure is completely described. However, except in very specific cases (see \cite{AM00}), its global topology remains a mystery. In particular, nothing is known about the global topology of the group of symplectomorphisms of the real four-dimensional torus. Thanks to Moser's lemma, it can be shown that the space of symplectic forms over the real two-dimensional torus ($\T^2$) is path connected. This implies that its group of symplectomorphisms for a fixed symplectic form is homotopic to a lattice (i.e. each of its connected components is contractible). This project aims to use a gradient flow to study the topology of this group, seeing it as a sort of Morse-Bott function. This gradient flow will come from a moment map whose zeros are exactly the symplectomorphisms. To realize this moment map, we use the natural flat HyperKähler structure that one can put on the real four-dimensional torus, and the fact that the degree two cohomology of this space can be completely described. We also first write this approach in the familiar and simpler setting of the real two-dimensional torus in order to build up the intuition and maybe see a certain partern appering that could lead to the study of any group of symplectomorphisms of even-dimensional tori.

\subsection{Notations}\label{nota}
An (Hyper-)Kähler torus is a quotient $V/\Lambda$, where
\begin{itemize}
    \item $V$ is a real two-dimensional vector space (resp. $4$-dimensional in the HyperKähler setting) ;
    \item $\Lambda$ is a lattice of $V$, and it acts on $V$ by translation ;
    \item $V$ is identified via a linear isomorphism to $\C$ (resp. the quaternionic space $\mathbb{H}$ in the HyperKähler setting).
\end{itemize}
\noindent
When the dimension is $2$, we will take $\T^2=\C/\Lambda$ for simplicity. And in dimension $4$, we will take $\T^4=\mathbb{H}/\Lambda$ with $\mathbb{H}=\{x_1+ix_2+jx_3+kx_4\,|\,(x_1,x_2,x_3,x_4)\in\R^4\}$ (recall the quaternionic relations $i^2=j^2=k^2=ijk=-1$).
\begin{remark}
\begin{enumerate}
    \item The quotient torus $\T^2$, as a quotient of the complex plane, possesses numerous structures coming from $\C$ that are preserved by the action by translation of the lattice $\Lambda$. In particular, it is a Riemannian manifold when endowed with the metric $g=|dz|^2$ on each fiber, and it is a Kähler manifold when endowed with the complex structure $I$, corresponding to the multiplication by $i$ in each fibers, and the $2$-form $\sigma=\frac{i}{2}dz\wedge d\Bar{z}$.
    \item The quotient torus $\T^4$ inherits: a Riemannian structure $g=|dq|^2$ inherited from $\mathbb{H}$ ; a complex structure $J$ corresponding on each tangent space to the multiplication on the right by $j$ on $\mathbb{H}$ (since we also have a complex structure acting by multiplication on the right by $j$ that we also want to write $J$, we will keep writing $JX$ or $XJ$ depending on which complex structure we think of) ; a Kähler structure $\omega\overset{\text{loc}}{=}dx_1\wedge dx_3+dx_2\wedge dx_4=g(\cdot J,\cdot)$ ; three other complex structures corresponding this time to the multiplication on the right by $i$, $j$ and $k$ on tangent spaces ; three other Kähler forms $\omega_I(\cdot,\cdot)=g(I\cdot,\cdot)$, $\omega_J(\cdot,\cdot)=g(J\cdot,\cdot)$ and $\omega_K(\cdot,\cdot)=g(K\cdot,\cdot)$ that together form a HyperKähler structure on $\T^4$. Note that each of the $\omega_\bullet$ for $\bullet\in\{I,J,K\}$ induces the opposite orientation on $\T^4$ than $\omega$.
\end{enumerate}
\end{remark}
\noindent
In what follows, we will be interested by the group of symplectomorphisms $\SympA$ of $\T^2$, and $\SympB$ of $\T^4$.

\subsection{Main Results}
In this subsection, we expose a version of our two main theorems. The gauge groups of symplectic Hamiltonian transformations 
$$
\begin{array}{ccc}
Ham(\T^2,\sigma) & \text{ and } & Ham(\T^4,\omega)
\end{array}
$$
\noindent
act by precomposition respectively on
$$
\begin{array}{ccc}
\operatorname{Diff}^+(\T^2)=\{f\in\operatorname{Diff}(\T^2)\,|\,f^*[\sigma]=[\sigma]\} & \text{ and } & \operatorname{Diff}_0(\T^4)=\{f\in\operatorname{Diff}(\T^4)\,|\,f\sim\operatorname{Id}\},
\end{array}
$$
\noindent
where $[\sigma]$ is the cohomological class associated to $\sigma$ and $f\sim\operatorname{Id}$ means that $f$ is homotopic to the identity. In the case of the two-dimensional torus, we show that $\operatorname{Diff}^+(\T^2)$ admits a Kähler structure as a Banach manifold as long as a Hamiltonian action by $Ham(\T^2,\sigma)$ (cf \cite{AT22} for the definitions related to Banach manifolds):
\begin{theorem}\label{thm12}
The space $\operatorname{Diff}^+(\T^2)$ admits a Kähler strucutre $(\operatorname{Diff}^+(\T^2),G,I)$, where $G$ is a scalar product compatible with the complex structure $I$. We denote by $\Omega$ the associated Kähler form. This form is invariant under the action by precomposition of $Ham(\T^2,\sigma)$. Moreover, the action of $Ham(\T^2,\sigma)$ is Hamiltonian with respect to $\Omega$, and the associated moment map is
$$
\mu:\left\{\begin{array}{ll}\operatorname{Diff}^+(\T^2)\rightarrow\mathcal{C}^\infty_0(\T^2)\\ f\mapsto-\frac{f^*\sigma}{\sigma}+1\end{array}\right.,
$$
\noindent
its zero locus is exactly the space of symplectomorphisms $\SympA$.
\end{theorem}
\noindent
We can now, using the ideas first introduced by Atiyah and Bott, consider the associated norm map 
$$
\phi:\left\{\begin{array}{ll}\operatorname{Diff}^+(\T^2)\rightarrow\R\\ f\mapsto\frac{1}{2}\|\mu(f)\|_{L^2}^2\end{array}\right..
$$
\noindent
The function $\phi$ may be seen as a Morse-Bott function, and the study of its gradient flow may contain topological information about its critical locus. Here is the result formulated in the setting of $\T^2$ using this approach :
\begin{theorem}\label{thm13}
For all $f\in\operatorname{Diff}^+(\T^2)$, there exists $\epsilon(f)>0$ and a unique family $\{f_t\}\subset\operatorname{Diff}^+(\T^2)$ (smooth in $t$) that solves the differential problem
\begin{equation*}
    \begin{cases}
    \frac{\partial}{\partial t}f_t=-\nabla\phi(f)\\
    f_0=f
    \end{cases}
\end{equation*}
\noindent
for all $t\in[0,\epsilon(f))$.
\end{theorem}
\begin{remark}
Note that no comment has been made regarding the regularity of the initial condition nor the solution, which is technically important. We refer to \ref{thmA} for the technical details.
\end{remark}
\noindent
The results obtained in the case of $\T^4$ are somewhat quite similar but tainted with this extra HyperKähler structure. 
\begin{theorem}\label{thm15}
The space $\operatorname{Diff}_0(\T^4)$ admits a HyperKähler structure
$$
(\operatorname{Diff}_0(\T^4),G,I,J,K),
$$
\noindent
where $G$ is a scalar product compatible with all the complex structures $I$, $J$ et $K$ (whether they act on the left or on the right). We denote by $\Omega_I$, $\Omega_J$ et $\Omega_K$ the Kähler forms associated to $I$, $J$ and $K$ acting on the right. This HyperKähler structure is invariant under the action by precomposition of $Ham(\T^4,\omega)$. Moreover, the action of $Ham(\T^4,\omega)$ is Hamiltonian with respect to each $\Omega_\bullet$, the moment map associated to $\Omega_\bullet$ is
$$
\mu_\mathcal{\bullet}(f):=-\frac{1}{\text{vol}}(f^*\omega_{\bullet}\wedge\omega)
$$
\noindent
where $\operatorname{vol}=-\frac{1}{2}\omega\wedge\omega\underset{loc}{=}dx_1\wedge dx_2\wedge dx_3\wedge dx_4$ is the volume form associated to $\omega$ ; the zeros of $\underline{\mu}:=(\mu_I,\mu_J,\mu_K)$ are exactly the symplectomorphims of $\SympB$.
\end{theorem}
\noindent
Introducing the same norm function as before
$$
\phi:\left\{\begin{array}{ll}\operatorname{Diff}_0(\T^4)\rightarrow\R\\ f\mapsto\frac{1}{2}\left(\|\mu_I(f)\|^2+\|\mu_J(f)\|^2+\|\mu_K(f)\|^2\right)\end{array}\right.
$$
\noindent
we study its associated gradient flow and obtain
\begin{theorem}\label{thm16}
For all $f\in\operatorname{Diff}_0(\T^4)$, there exists $\epsilon(f)>0$ and a unique family $\{f_t\}\subset\operatorname{Diff}_0(\T^4)$ (smooth in $t$) that solves the differential problem
\begin{equation*}
    \begin{cases}
    \frac{\partial}{\partial t}f_t=-\nabla\phi(f_t)\\
    f_0=f
    \end{cases}
\end{equation*}
\noindent
for all $t\in[0,\epsilon(f))$.
\end{theorem}
\noindent
A more precise formulation of this result can be found at the end of the second part of this paper \ref{thmB}.

\section{Case of $\T^2$}
\subsection{Setting and properties}\label{settT2}
\noindent
In this section, we fix $\T^2=\C/\Lambda$ and $\sigma\overset{\text{loc}}{=}\frac{i}{2}dz\wedge d\Bar{z}$. Consider the Hölder topologie on the group of diffeomorphisms over a compact manifold. Endow the following spaces with this topology:
\begin{definition}
\label{def21}
Denote by $\mathcal{M}_{k,\alpha}$ the space of $\mathcal{C}^{k,\alpha}$ diffeomorphisms of $\T^2$ preserving the orientation and $\mathcal{G}_{k,\alpha}:=Ham^{(k,\alpha)}(\T^2,\sigma)\subset\SympA^{(k,\alpha)}$ the subgroup of $\mathcal{C}^{k,\alpha}$ symplectic Hamiltonian transformations. This latter subgroup acts on $\mathcal{M}_{k,\alpha}$ by precomposition (as long as $\SympA^{(k,\alpha)}$). Those spaces can be respectively considered as a $\mathcal{C}^{k,\alpha}$ Banach manifold and a $\mathcal{C}^{k,\alpha}$ Banach Lie group (cf \cite{AT22}). The tangent space to an element $f\in\mathcal{M}_{k,\alpha}$ is the vector space of global $\mathcal{C}^{k,\alpha}$ section of the pullback bundle $f^*T\T^2$, i.e. the $\mathcal{C}^{k,\alpha}$ applications $u$ that make the following diagram commute:
$$
\xymatrix{
    & T\T^2=\T^2\times\R^2 \ar[d] \\
    \T^2 \ar[ur]^u \ar[r]_f & \T^2
    }
$$
\noindent
For $u,v\in T_f\mathcal{M}=\Gamma(f^*T\T^2)$, we also define two additional structures:
$$
G(u,v)=\int_{\T^2}g(u,v)\sigma,
$$
\noindent
and :
$$
\Omega(u,v)=\int_{\T^2}\sigma(u,v)\sigma=\int_{\T^2}g(Iu,v)\sigma=G(Iu,v),
$$
\noindent
where the action of the complex structure $I$ on $f^*T\T^2$ is given point-wise by: $(Iu)_x=I_{f(x)}u_x$.
\end{definition}
\begin{remark}
\noindent
\begin{enumerate}
    \item We note that the sections of $f^*T\T^2$ can be written in a unique way in the form $f_*X$, where $X$ is a honest vector field (i.e. a global section of $T\T^2$). In the sequel of this paper, we will always use this notation: $f_*X\in T_f\mathcal{M}$.
    \item The complex structure $I$ also acts on the vector fields:
    $$
    \forall X\in T\T^2,\;\forall x\in\T^2,\;(IX)_x=I_xX_x.
    $$
    \item The formulae given to define $G$ can be used to get a scalar product on any tensor above $\T^2$. We shall therefore use this notation for any type of tensors. Moreover, in order to simplify the notations a little, we'll drop the indices $k$ and $\alpha$ when there is no need to specify them.
\end{enumerate}
\end{remark}
\begin{proposition}\label{prop23}
$\Omega$ is a symplectic form on the Banach manifold $\mathcal{M}$, and $G$ is a Riemannian structure on it. Moreover, both are $\SympA$-invariant.
\end{proposition}
\begin{proof}
Recall the definitions of differential forms, exterior differential, and Lie differential on a Banach manifold in \cite{AT22} paragraph 1.9. Then, it immediatly follows that the exterior differential on $\Omega$ commutes with the integral and acts on the form $\sigma$ exactly as the exterior differential of $\T^2$: $d\Omega=d\left(\int_{\T^2}\sigma(.,.)\sigma\right)=\int_{\T^2}d\sigma(.,.)\sigma=0$ (because $d\sigma=0$). Similarly, we see that $\Omega$ is non-degenerated. The non-degeneratedness of $G$ is also derived quite easly. Let us now show that both are $\SympA$-invariant. For $\varphi\in\SympA$, we define:
$$
\psi_\varphi:\left\{\begin{array}{ll}\mathcal{M}\rightarrow\mathcal{M}\\f\mapsto f\circ\varphi\end{array}\right.,
$$
\noindent
and compute its differential:
$$
D_f\psi_\varphi:\left\{\begin{array}{ll}T_f\mathcal{M}\rightarrow T_{f\circ\varphi}\mathcal{M}\\v\mapsto v\circ\varphi\end{array}\right..
$$
\noindent
The pullback of $\Omega$ by this action is:
\begin{align*}
\left(\psi_\varphi^*\Omega\right)(u,v)&=\Omega(D_f\psi_\varphi(u),D_f\psi_\varphi(v))=\int_{\T^2}\sigma(u\circ\varphi,v\circ\varphi)\sigma\\&=\int_{\T^2}\varphi^*\sigma(u,v)\sigma=\int_{\T^2}\varphi^*\sigma(u,v)\varphi^*\sigma\\
&=\int_{\T^2}\sigma(u,v)\sigma=\Omega(u,v),
\end{align*}
\noindent
where $\sigma(u,v)$ is viewed as a function on $\T^2$, and we used that $\varphi^*\sigma=\sigma$. The computations for $G$ are identical given that $(Iu)\circ\varphi=I(u\circ\varphi)$ and $G(Iu,v)=\Omega(u,v)$.
\end{proof}
\noindent
We thus defined a (weak) symplectic Banach manifolds $(\mathcal{M},\Omega)$ that we'll now use to produce the desired moment map.

\subsection{Moment map and gradient flow}
\noindent
\begin{theorem}\label{thm24}
The following map is a moment map for the action of $\mathcal{G}$ on $\mathcal{M}$ (\ref{def21}):
$$
\mu:\left\{\begin{array}{ll}\mathcal{M}\rightarrow\mathcal{C}_0^{k,\alpha}(\T^2)\\f\mapsto -H_f+1\end{array}\right.,
$$
\noindent
where $H_f$ is the unique function such that $f^*\sigma=H_f\sigma$ (we can also write $H_f=\frac{f^*\sigma}{\sigma}$). Moreover, its zeros are exactly the symplectomorphisms.
\end{theorem}
\begin{proof}
The infinitesimal action of $\mathcal{G}$ on $\mathcal{M}$ is given by the pushforwards of Hamiltonian vector fields with respect to $\sigma$. Let $f\in\mathcal{M}$ and $\{\varphi_t\}$ be a Hamiltonian isotopy such that $\left.\frac{d}{ds}\right|_{s=0}\varphi_s=X_h\in\Gamma_{ham}(T\T^2)$. We then have:
$$
\left.\frac{\partial}{\partial s}\right|_{s=0}\left(f\circ\varphi_s\right)=f_*\left(\left.\frac{\partial}{\partial s}\right|_{s=0}\varphi_s\right)=f_*X_h.
$$
\noindent
Now, let $X\in\Gamma_{symp}(T\T^2)$ and $f_*Y\in T_f\mathcal{M}$. We get:
\begin{align*}
\Omega(f_*X_h\,,\,f_*Y)&=\int_{\T^2}f^*\sigma(X_h\,,\,Y)\sigma=-\int_{\T^2}(\iota_{X_h}\iota_Yf^*\sigma)\sigma\\&=-\int_{\T^2}\iota_Yf^*\sigma\wedge\iota_{X_h}\sigma=\int_{\T^2}\iota_Yf^*\sigma\wedge dh\\
&=-\int_{\T^2}dh\wedge\iota_Yf^*\sigma=\int_{\T^2}hd\iota_Yf^*\sigma\\
&=\int_{\T^2}h\left(\frac{d\iota_Yf^*\sigma}{\sigma}\right)\sigma=-G\left(-\frac{d\iota_Yf^*\sigma}{\sigma},h\right),
\end{align*}
\noindent
where we used that: the interior product and the exterior differential are anti-derivations (in particular, at the second line: $\iota_{X_h}(\iota_Yf^*\sigma\wedge\sigma)=\iota_{X_h}\iota_Yf^*\sigma\wedge\sigma-\iota_Yf^*\sigma\wedge\iota_{X_h}\sigma$) ; that the space of $3$-forms on a surface is reduced to $\{0\}$ ; and Stokes theorem. We then show that the differential of $\mu$ at $f\in\mathcal{M}$ is equal to $-\frac{d\iota_Yf^*\sigma}{\sigma}$. But this comes from the Cartan formulae and the closedness of $\sigma$:
$$
D_f\mu(f_*Y)=\left.\frac{\partial}{\partial  s}\right|_{s=0}\left(-\frac{f_s^*\sigma}{\sigma}+1\right)=-\frac{\left.\frac{\partial}{\partial  s}\right|_{s=0}\left(f_s^*\sigma\right)}{\sigma}=-\frac{\mathcal{L}_Yf^*\sigma}{\sigma}=-\frac{d\iota_Yf^*\sigma}{\sigma}
$$
\noindent
The only remaining thing to prove is that $\mu$ is $\mathcal{G}$-equivariant. We show something even stronger: that it is $\SympA$-equivariant. Let $\varphi\in\SympA$ and $f\in\mathcal{M}$, we have:
\begin{align*}
    \mu(f\circ\varphi)&=-\frac{(f\circ\varphi)^*\sigma}{\sigma}+1=-\frac{\varphi^*f^*\sigma}{\varphi^*\sigma}+1\\&=-\varphi^*\left(\frac{f^*\sigma}{\sigma}\right)+1=\varphi^*\left(-\frac{f^*\sigma}{\sigma}+1\right)=\mu(f)\circ\varphi.
\end{align*}
\end{proof}
\begin{definition}\label{def25}
We intoduce the following norm function:
$$
\phi:\left\{\begin{array}{ll}\mathcal{M}\rightarrow \R\\f\mapsto\frac{1}{2}\|\mu(f)\|^2:=\frac{1}{2}G(\mu(f),\mu(f))=\frac{1}{2}\int_{\T^2}\mu(f)^2\sigma\end{array}\right.
$$
\noindent
as well as its associated gradient flow:
\begin{equation}\label{FG}
\frac{\partial \Tilde{f_t}}{\partial t}=-\nabla\phi(\Tilde{f_t}),
\end{equation}
\noindent
where the gradient is taken with respect to the Riemannian metric $G$:
$$
D_f\phi(f_*Y)=G(\nabla\phi(f),f_*Y),
$$
\noindent
for all $f\in\mathcal{M}$ and $f_*Y\in T_f\mathcal{M}$.
\end{definition}
\noindent
We enumerate some of the properties of that flow:
\begin{proposition}\label{prop26} Using the same notations as before:
\begin{enumerate}
    \item $\phi$ is $\SympA$-invariant ;
    \item $\nabla\phi(f)=f_*\sharp_\sigma (H_f)^2d^*f^*\sigma$ (where $H_f$ is defined in \ref{thm24} and the operator $\sharp_\sigma$ is the musical isomorphism with respect to $\sigma$ that sends $1$-forms to vector fields) ;
    \item Crit($\phi$)=$\phi^{-1}(0)$ ;
    \item $\nabla\phi$ is $\SympA$-equivariant ;
    \item $D_f\nabla\phi(f_*Y)=\left(\left.\frac{\partial}{\partial s}\right|_{s=0}(f_s)_*\sharp_\sigma (H_f)^2d^*f^*\sigma\right)+f_*\sharp_\sigma 2\frac{d\iota_Yf^*\sigma}{\sigma}H_fd^*f^*\sigma+f_*\sharp_\sigma (H_f)^2d^*d\iota_Yf^*\sigma$.
\end{enumerate}
\end{proposition}
\begin{proof}
\noindent
\begin{enumerate}
    \item $2\phi(f\circ\varphi)=\int_{\T^2}(\mu(f\circ\varphi))^2\sigma=\int_{\T^2}(\mu(f)\circ\varphi)^2\varphi^*\sigma=\int_{\T^2}\varphi^*\left((\mu(f))^2\sigma\right)=2\phi(f)$ ;
    \item We compiute the differential of $\phi$ at an element $f\in\mathcal{M}$:
    \begin{align*}    D_f\phi(f_*Y)&=G\left(D_f\mu(f_*Y),\mu(f)\right)=G\left(d\iota_Yf^*\sigma,f^*\sigma\right)\\    &=G\left(\iota_Yf^*\sigma,d^*f^*\sigma\right)=G\left(H_f\iota_Y\sigma,d^*f^*\sigma\right)=G\left(Y,\sharp_\sigma H_fd^*f^*\sigma\right)\\
    &=\int_{\T^2}g(Y\,,\,\sharp_\sigma H_fd^*f^*\sigma)\sigma=\int_{\T^2}f^*\left(g(Y\,,\,\sharp_\sigma H_fd^*f^*\sigma)\sigma\right)\\
    &=\int_{\T^2}f^*g(Y\,,\,\sharp_\sigma H_fd^*f^*\sigma)f^*\sigma=\int_{\T^2}f^*g(Y\,,\,H_f\sharp_\sigma H_fd^*f^*\sigma)\sigma\\
    &=G(f_*Y,f_*\sharp_\sigma (H_f)^2d^*f^*\sigma),
    \end{align*}
    \noindent
    where we used at the second line:
    \begin{itemize}
    \item $\iota_Yf^*\sigma=\iota_YH_f\sigma=H_f\iota_Y\sigma$ ;
    \item $\flat_\sigma(Y)=\iota_Y\sigma=\iota_{iY}g=\flat_g(iY)=i\flat_g(Y)$ ;
    \item $\sharp_\sigma=-i\sharp_g$ ;
    \item the fact that $\flat_g$ and $\sharp_g$ are dual operators for the metric $G$ ;
    \item the fact that $G$ is $I$-invariant: $G\left(iY,\sharp_g H_fd^*f^*\sigma\right)=G\left(Y,\sharp_\sigma H_fd^*f^*\sigma\right)$.
    \end{itemize}
    \noindent
    We deduce the expression for its gradient: $\nabla\phi(f)=f_*\sharp_\sigma (H_f)^2d^*f^*\sigma$.
    \item The differential of $\phi$ was computed in the previous point:
    $$
    D_f\phi:\left\{\begin{array}{ll}T_f\mathcal{M}\rightarrow\R\\f_*Y\mapsto G(D_f\mu(f_*Y),\mu(f))=G(f_*Y,f_*\sharp_\sigma(H_f)^2d^*f^*\sigma)\end{array}\right..
    $$
    \noindent
    Therefore, since $G$ is a Riemannian metric:
    $$
    \forall\,f_*Y\in T_f\mathcal{M},\;D_f\phi(f_*Y)=0\iff f_*\sharp_\sigma(H_f)^2d^*f^*\sigma=0\iff d^*f^*\sigma=0.
    $$
    \noindent
    The last equivalence come from the fact that all terms in front of $d^*f^*\sigma$ are invertible linear operators. Thus, $f^*\sigma$ is harmonic. Moreover, since $f$ preserves the orientation of $\T^2$, it verifies
    $$
    [f^*\sigma]=f^*[\sigma]=[\sigma],
    $$
    \noindent
    i.e. $f^*\sigma-\sigma=d\alpha$ for some $1$-form $\alpha$. We then see that if $d^*d\alpha=0$ we must have $d\alpha=0$ on $\T^2$:
    $$  \|d\alpha\|^2=\int_{\T^2}g(d\alpha,d\alpha)\sigma=\int_{\T^2}g(\alpha,d^*d\alpha)\sigma=0.
    $$
    Hence $f\in\SympA=\phi^{-1}(0)$. On the other hand, if $f\in\SympA$, we get $d^*f^*\sigma=d^*\sigma=0$, because $\sigma$ is harmonic (it is an area form), hence $D_f\phi=0$, i.e. $f\in$ Crit($\phi$).
    \item Let $f\in\mathcal{M}$ and $\varphi\in\SympA$. Starting from the $\SympA$-invariance of $\phi$, we get:
\begin{align*}
\phi(f\circ\varphi)=\phi(f)&\implies\forall\,v\in T_f\mathcal{M},\; D_{f\circ\varphi}\phi(v\circ\varphi)=D_f\phi(v)\\
&\implies\forall\,v\in T_f\mathcal{M},\;G\left(\nabla\phi(f\circ\varphi),v\circ\varphi\right)=G\left(\nabla\phi(f),v\right)\\
&\implies\forall\,v\in T_f\mathcal{M},\;G\left(\nabla\phi(f\circ\varphi),v\circ\varphi\right)=G\left(\nabla\phi(f)\circ\varphi,v\circ\varphi\right)\\
&\implies\forall\,w\in T_{f\circ\varphi}\mathcal{M},\;G\left(\nabla\phi(f\circ\varphi),w\right)=G\left(\nabla\phi(f)\circ\varphi,w\right)\\
&\implies\nabla\phi(f\circ\varphi)=\nabla\phi(f)\circ\varphi,
\end{align*}
\noindent
where we used at the third line that $G\left(\cdot,\cdot\right)$ is $\SympA$-invariante.
\item It is a direct computation using the Leibniz rule:
\begin{align*}
D_f\nabla\phi(f_*Y)&=\left.\frac{\partial}{\partial s}\right|_{s=0}\left(\nabla\phi(f_s)\right)\\
&=\left.\frac{\partial}{\partial s}\right|_{s=0}\left({f_s}_*\sharp_\sigma (H_{f_s})^2d^*{f_s}^*\sigma\right)\\
&=\left(\left.\frac{\partial}{\partial s}\right|_{s=0}(f_s)_*\sharp_\sigma (H_f)^2d^*f^*\sigma\right)+f_*\sharp_\sigma 2\frac{d\iota_Yf^*\sigma}{\sigma}H_fd^*f^*\sigma+f_*\sharp_\sigma (H_f)^2d^*d\iota_Yf^*\sigma,
\end{align*}
\noindent
where, for the last line, we used that:
\begin{align*}
\left.\frac{\partial}{\partial s}\right|_{s=0}H_{f_s}&=\left.\frac{\partial}{\partial s}\right|_{s=0}\frac{f_s^*\sigma}{\sigma}=\frac{\left.\frac{\partial}{\partial s}\right|_{s=0}f_s^*\sigma}{\sigma}=\frac{\mathcal{L}_Yf^*\sigma}{\sigma}=\frac{d\iota_Yf^*\sigma}{\sigma}.
\end{align*}
\end{enumerate}
\end{proof}
\begin{remark}
Note that for all $f\in\mathcal{M}$ the highest order term of the linear differential operator $D_f\nabla\phi$ is the last one. Hence it is not elliptic since its kernel contains a vector space isomorphic to the space of 1-forms. The kernel hence is infinite dimensional, which isn't that surprising given the invariance of the operator under the gauge group. But we shall show that a DeTurck type of argument (using this invariance) produces a second equation that is parabolic and whose solutions are solutions of the initial problem.
\end{remark}

\subsection{DeTurck Trick}
\noindent
In this subsection, we shall explore the trick discovered by DeTurck to produce out of an invariant non-parabolic equation a parabolic one whose solutions are solutions of the first one. In this part, we shall use the notation $\partial_t=\frac{\partial}{\partial t}$.
\begin{lemma}\label{lem28}
Let $\{\Tilde{f_t}\}$ be a solution of \eqref{FG} and $\{\varphi_t\}$ be a one-parameter family of Hamiltonian symplectomorphsisms. Then:
$$
\partial_t\left(\Tilde{f_t}\circ\varphi_t\right)=-\nabla\left(\Tilde{f_t}\circ\varphi_t\right)+(\Tilde{f_t})_*(\partial_t\varphi_t).
$$
\end{lemma}
\noindent
\begin{proof}
Using the $\SympA$-equivariance of $\nabla\phi$ exhibited in \ref{prop26}, we get:
$$
\partial_t(\Tilde{f_t}\circ\varphi_t)=(\partial_t\Tilde{f_t})\circ\varphi_t+(\Tilde{f_t})_*(\partial_t\varphi_t)=-\nabla\phi(\Tilde{f_t})\circ\varphi_t+(\Tilde{f_t})_*(\partial_t\varphi_t)=-\nabla(\Tilde{f_t}\circ\varphi_t)+(\Tilde{f_t})_*(\partial_t\varphi_t).
$$
\end{proof}
\noindent
Hence, we see that one can act on the flow by mean of a one-parameter family of Hamiltonian symplectomorphisms in order to absorb a term of the form $(\Tilde{f_t})_*W_t$, where $\{W_t\}$ is a one-parameter family of Hamiltonian vector field. We then introduce a family of equations, depending on such a family of Hamiltonian vector field, among which we shall exhibit a parabolic equation.
\begin{definition}\label{def29}
We call the following equation:
\begin{equation}\label{presqFGm}
\partial_tf_t=-\nabla\phi(f_t)+(f_t)_*W_t,
\end{equation}
\noindent
the gradient flow associated to the family $\{W_t\}$ of Hamiltonian vector field.
\end{definition}
\begin{proposition}\label{prop210}
Let $\{f_t\}$ be a solution of \eqref{presqFGm}.
\noindent
\begin{enumerate}
    \item 
    Let $\{\varphi_t\}$ be a one-parameter family of symplectomorphisms. Then:
    $$
    \partial_t(f_t\circ\varphi_t)=-\nabla\phi(f_t\circ\varphi_t)+(f_t)_*(W_t\circ\varphi_t+\partial_t\varphi_t)\;;
    $$
    \item The ordinary differential equation
    \begin{equation}\label{ODE1}
    \begin{cases}
    W_s\circ\varphi_s+\partial_s\varphi_s=0\\
    \varphi_0=Id
    \end{cases}
    \end{equation}
    \noindent
    admits a unique solution $\{\varphi_s\}$ (constituted of symplectomorphisms). Moreover, such a solution induces a solution to \eqref{FG}.
\end{enumerate}
\end{proposition}
\begin{proof}
\noindent
\begin{enumerate}
\item If $\{f_t\}$ is a solution of \eqref{presqFGm}, the $\SympA$-equivariance of $\nabla\phi$ gives:
\begin{align*}
\partial_t(f_t\circ\varphi_t)&=(\partial_tf_t)\circ\varphi_t+(f_t)_*(\partial_t\varphi_t)\\
&=(-\nabla\phi(f_t)+(f_t)_*W_t)\circ\varphi_t+(f_t)_*(\partial_t\varphi_t)\\
&=-\nabla\phi(f_t\circ\varphi_t)+(f_t)_*(W_t\circ\varphi_t+\partial_t\varphi_t).
\end{align*}
\item The inverse of the flow o $W_s$ is the unique solution to this ODE (more details about the existence and the uniqueness of such flows can be found at the paragraph 3.1 of the third chapter of \cite{CK04}). This solution hence exists in arbitrary long times since $\T^2$ is compact. The last statement of the proposition immediately follows from the first point: just take $\Tilde{f_t}=f_t\circ\varphi_t$. We note that $\{\Tilde{f_t}\}$ exists as soon as $\{f_t\}$ does.
\end{enumerate}
\end{proof}
\noindent
It remains to find the family $\{W_t\}$ that will make the problem \eqref{presqFGm} parabolic. Here is a good anzats.
\begin{definition}\label{def134} Let $f\in\mathcal{M}$. Let
$$
W(f):=-\sharp_\sigma dd^*\iota_{F}\sigma_0,
$$
\noindent
where $F:\R^2\rightarrow\R^2$ is the lift of $f:\T^2\rightarrow\T^2$ (unique up to the action by translation by the lattice), and $\sigma_0$ is the standard sympletic form on $\R^2$. We thus get the modified equation
\begin{equation}\label{FGm}
    \partial_tf_t=-\nabla\phi(f_t)+(f_t)_*W(f_t)
\end{equation}
\end{definition}
\begin{remark}
Note that the form $\sigma$ is induced by $\sigma_0$ (cf \ref{settT2}). Therefore, modulo the point-wise standard identification $T_p\T^2\cong\R^2$ for $p\in\T^2$, we can view vector fields on $\T^2$ as applications $\R^2\rightarrow\R^2$.
\end{remark}
\begin{proposition}\label{prop213}
\noindent
\begin{enumerate}
\item For all diffeomorphism preserving the orientation $f$, $W(f)$ is a Hamiltonian vector field ;
\item $\forall\,f_*Y\in T_f\mathcal{M},\;D_fW(f_*Y)=-\sharp_\sigma dd^*\iota_Yf^*\sigma$ ;
\item The linearisation of \eqref{FGm} is $-D_f\nabla\phi(f_*Y)+\partial_s\left((f_s)_*W(f_0)\right)+f_*\partial_sW(f_s)$, and its highest order term is: 
$$
-f_*\sharp_\sigma((H_f)^2d^*d+dd^*)\iota_Yf^*\sigma\;;
$$
\item The linearisation of the right hand-side of \eqref{FGm} is elliptic, hence the problem is parabolic.
 \end{enumerate}
\end{proposition}
\begin{remark}
Add a comment about (strong) parabolicity for differential equations on a compact manifold and its consequences. $+$ No need to specify that the ellipticity is strong since on a compact manifold, both the ellipticity and the strong ellipticity are equivalent.
\end{remark}
\begin{proof}
\noindent
\begin{enumerate}
\item By definition of $W$, and using that $\flat_\sigma$ is the inverse isomorphism of $\sharp_\sigma$, we get:
$$
\iota_{W(f)}\sigma=\flat_\sigma(W(f))=\flat_\sigma(-\sharp_\sigma dd^*\iota_{F}\sigma)=-dd^*\iota_{F}\sigma.
$$
\noindent
Hence $W(f)$  is the Hamiltonian vector field associated to the Hamiltonian function $d^*\iota_{F}\sigma$.
\item $$
\partial_sW(f_s)=-\partial_s(\sharp_\sigma dd^*\iota_{F_s}\sigma)=-\sharp_\sigma dd^*\partial_s(\iota_{F_s}\sigma)=-\sharp_\sigma dd^*\iota_Yf^*\sigma\text{ où }\partial_sF_s=f_*Y.
$$
\item Recall \eqref{FGm}:
$$
\partial_tf_t=-\nabla\phi(f_t)+(f_t)_*W(f_t).
$$
\noindent
The linearisation of this equation is obtained through the computation of the following quantity:
$$
\left.\frac{\partial}{\partial s}\right|_{s=0}\left(-\nabla\phi(f_s)+(f_s)_*W(f_s)\right),
$$
\noindent
where $\{f_s\}$ is thought of as a perturbation of the equation by a one-parameter family of elements in $\mathcal{M}$ (i.e. $\partial_sf_s=f_*Y$). The formulae follows from the Leibniz rule. Now, since the highest order term of $-\partial_s\nabla\phi(f_s)$ is $-f_*\sharp_\sigma(H_f)^2d^*d\iota_Yf^*\sigma$, and that that of $f_*\partial_sW(f_s)$ is $-f_*\sharp_\sigma dd^*\iota_Yf^*\sigma$ (both are degree 2 terms), we obtain:
$$
-f_*\sharp_\sigma(H_f)^2d^*d\iota_Yf^*\sigma-f_*\sharp_\sigma dd^*\iota_Yf^*\sigma=-f_*\sharp_\sigma((H_f)^2d^*d+dd^*)\iota_Yf^*\sigma.
$$
\noindent
Those expressions must be read point-wise.
\item Whenever $f$ is a diffeomorphism preserving the identity, we have: $H_f>0$. We show that the operator $L:=(H_f)^2d^*d+dd^*$ is elliptic using its principal symbol ; the operator $-f_*\sharp_\sigma$ being invertible, it plays no role. Let $\alpha\in\Omega^1(\T^2)$. Choose a local chart of $\T^2$ and write: 
$$
\alpha=\alpha_xdx+\alpha_ydy,
$$
\noindent
where $\{dx,dy\}$ is an orthonormal basis of $T^*\T^2$ (since $T^*\T^2\cong\T^2\times\R^2$, we can pick the canonical basis of $\R^2$), and $\alpha_x$ and $\alpha_y$ are smooth functions of $\T^2$. We now describe the action of $L$ on $\alpha$.\\
\noindent
\begin{enumerate}
\item
\begin{align*}
dd^*\alpha&=d*d*\alpha=d*d(\alpha_xdy-\alpha_ydx)=d*[(\partial_x\alpha_x+\partial_y\alpha_y)dx\wedge dy]\\
&=d(\partial_x\alpha_x+\partial_y\alpha_y)=(\partial_x^2\alpha_x+\partial_x\partial_y\alpha_y)dx+(\partial_y\partial_x\alpha_x+\partial_y^2\alpha_y)dy.
\end{align*}
\item
\begin{align*}
d^*d\alpha&=*d*d\alpha=*d*[(-\partial_y\alpha_x+\partial_x\alpha_y)dx\wedge dy]=*d(-\partial_y\alpha_x+\partial_x\alpha_y)\\
&=*[(\partial_x^2\alpha_y-\partial_x\partial_y\alpha_x)dx+(\partial_y\partial_x\alpha_y-\partial_y^2\alpha_x)dy]\\
&=(\partial_x^2\alpha_y-\partial_x\partial_y\alpha_x)dy-(\partial_y\partial_x\alpha_y-\partial_y^2\alpha_x)dx\\
&=(\partial_y^2\alpha_x-\partial_y\partial_x\alpha_y)dx+(\partial_x^2\alpha_y-\partial_x\partial_y\alpha_x)dy.
\end{align*}
\item
\begin{align*}
(H_f^2d^*d+dd^*)\alpha&=H_f^2(\partial_y^2\alpha_x-\partial_y\partial_x\alpha_y)dx+H_f^2(\partial_x^2\alpha_y-\partial_x\partial_y\alpha_x)dy\\
&+(\partial_x^2\alpha_x+\partial_x\partial_y\alpha_y)dx+(\partial_y\partial_x\alpha_x+\partial_y^2\alpha_y)dy\\
&=(\partial_x^2\alpha_x+H_f^2\partial_y^2\alpha_x+(1-H_f^2)\partial_y\partial_x\alpha_y)dx\\
&+(\partial_y^2\alpha_y+H_f^2\partial_x^2\alpha_y+(1-H_f^2)\partial_y\partial_x\alpha_x)dy,
\end{align*}
\noindent
where we used Schwarz's lemme on the two functions $\alpha_x$ and $\alpha_y$ ($\partial_x\partial_y=\partial_y\partial_x$).
\end{enumerate}
\noindent
We can now write the principal symbol of $L=H_f^2d^*d+dd^*$:
$$
\sigma(L)(\xi)=\left(\begin{array}{cc}\xi_x^2+H_f^2\xi_y^2&(1-H_f^2)\xi_x\xi_y\\(1-H_f^2)\xi_y\xi_x&H_f^2\xi_y^2+\xi_x^2
\end{array}\right),
$$
\noindent
where $\xi=(\xi_x,\xi_y)\in T^*_p\T^2\cong\R^2$. The operator is elliptic since the principal symbol is invertible:
\begin{align*}
\operatorname{det}\left(\begin{array}{cc}\xi_x^2+H_f^2\xi_y^2&(1-H_f^2)\xi_x\xi_y\\(1-H_f^2)\xi_y\xi_x&H_f^2\xi_y^2+\xi_x^2
\end{array}\right)&=(\xi_x^2+H_f^2\xi_y^2)(H_f^2\xi_y^2+\xi_x^2)-(1-H_f^2)^2\xi_x^2\xi_y^2\\
&=H_f^2\xi_x^4+(1+H_f^4)\xi_x^2\xi_y^2+H_f^2\xi_y^4-(1-2H_f^2+H_f^4)\xi_x^2\xi_y^2\\
&=H_f^2(\xi_x^2+\xi_y^2)^2,
\end{align*}
\noindent
which is non-zero for all non-zero $\xi=(\xi_x,\xi_y)\neq(0,0)$ since $H_f^2>0$.
\end{enumerate}
\end{proof}

\subsection{Theorem for $\T^2$}\label{thmA}
\begin{theorem}
Let $k\in\mathbb{N}$ and $\alpha\in(0,1)$. For all $f\in\mathcal{M}^{k,\alpha}$, there exists $\epsilon(f)>0$ and a unique one-parameter family $\{\Tilde{f_t}\}$ of $\mathcal{M}^{k,\alpha}$ satisfying $\Tilde{f_0}=f$ and that solves the problem \ref{FG} for all $t\in[0,\epsilon)$. Moreover any $\Tilde{f_t}$ is smooth.
\end{theorem}
\begin{proof}
The proof is decomposed into 4 pieces.\\
\noindent
\begin{enumerate}
\item The equation \eqref{FGm} is parabolic by proposition \ref{prop213} so the result is certainly true for this equation with initial condition $f_0=f\in\mathcal{M}^{k,\alpha}$: a unique solution $\{f_t\}$ of \eqref{FGm} exists on a time interval $[0,\epsilon(f))$ (usual local existence theorem for strongly parabolic flow cf \cite{AL00} chapter 7 and 8).
\item We now solve the ODE \eqref{ODE1} for the family $\{W(f_t)\}$. It was explained in \ref{prop210} that since $\T^2$ is compact, a solution of this equation is unique and exists for as long as $\{f_t\}$ exists. Moreover, it is as regular as $\{f_t\}$.
\item Hence, the proposition \ref{prop213} gives a solution of \eqref{FG}: $\{\Tilde{f_t}:=f_t\circ\varphi_t\}$. This solution exists for all $t\in[0,\epsilon(f))$ and must be as regular as the least regular of the two applications of the composition (hence as regular as $\{f_t\}$).
\item We end the discussion with the regularity. The elliptic regularity for compact manifolds implies that the family $\{f_t\}$ is smooth by bootstrapping, and since the solution of \eqref{ODE1} is as regular as the family $\{f_t\}$, it follows that $\{\Tilde{f_t}\}$ is smooth as a composition of smooth functions.
\end{enumerate}
\end{proof}

\section{Case of $\T^4$}
\noindent
The case of $\T^4$ follows exactly the same pattern as that of $\T^2$. We shall see however that the moment map picture comes from a different structure: the HyperKähler structure on $\T^4$. 
\subsection{Setting and properties}\label{settT4}
\noindent
Let us describe the setting we will work with throughout this section. Recall that we equip $\T^4$ with forms $\omega$, $\omega_I$, $\omega_J$ and $\omega_K$ (cf \ref{nota}).
\begin{definition}\label{def31}
We equip $T\T^4$ with the following scalar product:
$$
G(\cdot,\cdot)=\int_{\T^4}g(\cdot,\cdot)\text{vol},
$$
\noindent
where $\text{vol}=-\frac{1}{2}\,\omega\wedge\omega$ is the usual volume form associated to the symplectic form $\omega$.
\end{definition}
\begin{remark}
This definition can be extended to any tensors, just like in the case of $\T^2$.
\end{remark}
\begin{definition}\label{def33}
The set of interest to construct the moment map here will be:
$$
\mathcal{M}:=\{f\in\text{Diff}(\T^4)\;|\;f\sim Id\}
$$
\noindent
where $f\sim Id$ means that $f$ is homotopic to the identity. The tangent space to $\mathcal{M}$ at an application $f$ is the space of parallel vector fields $u$ making the following diagram commute:
$$
\xymatrix {
    & T\T^4=\T^4\times\mathbb{H} \ar[d]^\pi \\
    \T^4 \ar[ur]^u \ar[r]_f & \T^4
  }
$$
\noindent
We can describe such parallel vector fields using a curve $\{f_t\}$ of $\mathcal{M}$ such that $\left\{\begin{array}{ll}\pdt{f_t}=u\\f_0=f\end{array}\right.$. In particular, we see that each $u\in T_f\mathcal{M}$ can be written in a unique way as $f_*X$, where $X$ is a (usual) vector field. The infinite dimensional Lie group $\mathcal{G}=Ham(\T^4,\omega)$ acts on $\mathcal{M}$ by precomposition.
\end{definition}
\begin{remark}
Note that we didn't specify in the above definition the regularity of the applications. We when ask all of them to be of class $\mathcal{C}^{k,\alpha}$ with $k\in\mathbb{N}$ and $\alpha\in(0,1)$, all spaces of sections and diffeomorphisms above are Banach manifolds (cf \cite{AT22} paragraph 1.7 page 14).
\end{remark}
\begin{proposition}\label{prop35}
The action of each complex structure can be extended on all vector fields point-wise. For example, the left multiplication by $i$ on $\mathbb{H}$ produces a complex structure $I$ acting on $T\T^4$ as: 
$$
\forall\;X\in\Gamma(T\T^4),\;\forall\;x\in\T^4,\;(IX)_x=I_xX_x=iX_x\,.
$$
\noindent
Let $\bullet\in\{I,J,K\}$, and define $\Omega_\bullet$ by:
$$
\Omega_\bullet(\cdot,\cdot)=\int_{\T^4}\omega_\bullet(\cdot,\cdot)vol=\int_{\T^4}g(\bullet\cdot,\cdot)vol=G(\bullet\cdot,\cdot).
$$
\noindent
Then $(\mathcal{M}\,,\,\Omega_{I}\,,\,\Omega_{J}\,,\,\Omega_{K})$ is a HyperKähler Banach manifold. Moreover, each $\Omega_\bullet$ is $\SympB$-invariant.
\end{proposition}
\begin{proof}
The proof that each $\Omega_\bullet$ is $\SympB$-invariant and produces a Kähler structure on $\mathcal{M}$ is exactly the same as in the case of $\T^2$ (cf \ref{prop23}). The fact that $\Omega_{I}$, $\Omega_{J}$ and $\Omega_{K}$ give a HyperKähler structure on $\mathcal{M}$ directly follows from the quaternionic relations between $I$, $J$ and $K$ on each tangent space.
\end{proof}
\noindent
Recall that $Lie(\mathcal{G})$ (the Lie algebra of $\mathcal{G}=Ham(\T^4,\omega)$) is isomorphic to ${\mathcal{C}}_0^\infty(\T^4)$. In order to determine the moment map associated to this action for one of the forms $\Omega_{\bullet}$, we shall find the unique solution (up to a constant) of the problem:
$$
\Omega_{\bullet}(f_*X_h\,,\,f_*Y)=-\LD{D_f\mu(f_*Y)}{h}
$$
where $h\in\mathcal{C}^\infty_0(\T^4)$, $f\in\mathcal{M}$, $f_*Y\in T_f\mathcal{M}$, and $X_h$ is the unique vector field such that $\iota_{X_h}\omega=-dh$ (the infinitesimal action of $\mathcal{G}$ on $\mathcal{M}$ is given by $f_*X_h$, for $f\in\mathcal{M}$ and $X_h\in\Gamma_{ham}(T\T^4,\omega)$).

\subsection{Moment map and gradient flow}
\noindent
\begin{theorem}\label{thm36}
The moment map associated to $\Omega_\bullet$ for the action of $\mathcal{G}$ on $\mathcal{M}$ (where $\bullet\in\{I,J,K\}$) is given by:
$$
\mu_\mathcal{\bullet}(f):=-\frac{1}{\text{vol}}(f^*\omega_{\bullet}\wedge\omega).
$$
\noindent
Together, they form a so-called HyperKähler moment map: $\underline{\mu}=(\mu_{I},\mu_{J},\mu_{K})$.
\end{theorem}
\begin{proof}
We detail the computations for $\mu_{I}$ ; those of $\mu_J$ and $\mu_K$ are identical.
$$
\Omega_{I}(f_*X_h\,,\,f_*Y)=\int_{\T^4}(f^*\omega_I)(X_h\,,\,Y)\text{vol}=-\int_{\T^4}(\iota_{X_h}\iota_Yf^*\omega_I)\text{vol}.
$$
\noindent
Since the contraction is an anti-derivation, we have:
$$
0=\iota_{X_h}(\iota_Yf^*\omega_I\wedge \text{vol})=(\iota_{X_h}\iota_Yf^*\omega_I)\text{vol}-(\iota_Yf^*\omega_I)\wedge(\iota_{X_h}\text{vol}).
$$
\noindent
On the other hand: $\iota_{X_h}\text{vol}=\iota_{X_h}\,(\frac{1}{2}\,\omega\wedge\omega)=\frac{1}{2}((\iota_{X_h}\omega)\wedge\omega+\omega\wedge(\iota_{X_h}\omega))=(\iota_{X_h}\omega)\wedge\omega$, so that:
$$
\Omega_{I}(f_*X_h\,,\,f_*Y)=-\int_{\T^4}(\iota_Yf^*\omega_I)\wedge\iota_{X_h}\omega\wedge\omega=\int_{\T^4}(\iota_Yf^*\omega_I)\wedge dh\wedge\omega=-\int_{\T^4}dh\wedge(\iota_Yf^*\omega_I)\wedge\omega.
$$
\noindent
We then apply Stokes theorem (with $\partial \T^4=\emptyset\implies\int_{\partial \T^4}=0$) to the form:
$$
d(h\iota_Yf^*\omega_I\wedge\omega)=dh\wedge(\iota_Yf^*\omega_I)\wedge\omega + hd(\iota_Yf^*\omega_I)\wedge\omega,
$$
\noindent
and we get:
$$
\Omega_{I}(f_*X_h\,,\,f_*Y)=\int_{\T^4} hd(\iota_Yf^*\omega_I)\wedge\omega=\int_{\T^4} h\frac{d(\iota_Yf^*\omega_I)\wedge\omega}{\text{vol}}\text{vol}=-\LD{-\frac{d(\iota_Yf^*\omega_I)\wedge\omega}{\text{vol}}}{h}.
$$
\noindent
The remain thing to be highlighted to finish this proof is: 
\begin{align*}
\partial_t(f_t^*\omega_I\wedge\omega)&=\partial_t((f\circ f^{-1}\circ f_t)^*\omega_I)\wedge\omega=\partial_t((f^{-1}\circ f_t)^*(f^*\omega_I))\wedge\omega\\
&=\mathcal{L}_Y(f^*\omega_I)\wedge\omega=(d\iota_Y+\iota_Yd)(f^*\omega_I)\wedge\omega=d(\iota_Yf^*\omega_I)\wedge\omega.
\end{align*}
\end{proof}
\noindent
We now show that the zero locus of $\underline{\mu}$ is exactly $\SympB$. To achieve this we recall some aspects of Hodge theory in real dimension four. Recall the star operator $*$ that acts on differential forms:
$$
\forall\;k\in\{1\,,\,2\,,\,3\,,\,4\},\;\forall\;\alpha\in\Omega^k(\T^4),\;\forall\;\beta\in\Omega^{k}(\T^4),\;\alpha\wedge*\beta=g(\alpha\,,\,\beta)vol.
$$
\noindent
The square of this operator satisfies:
$$
\forall\;k\in\{1\,,\,2\,,\,3\,,\,4\},\;\forall\;\alpha\in\Omega^k(\T^4),\;**\alpha=(-1)^{k(4-k)}\alpha.
$$
\noindent
For $k=2$, we get: $\left.*\right|_{\Omega^2(\T^4)}^2=Id$, which produces a $g$-orthogonal decomposition of $\Omega^2(\T^4)$ in eigen-spaces:
$$
\Omega^2(\T^4)=\Omega^2_+(\T^4)\oplus^{\perp_g}\Omega^2_-(\T^4)
$$
\noindent
where $\Omega^2_\pm(\T^4):=\{\alpha\in\Omega^2(\T^4)\;|\;*\alpha=\pm\alpha\}$. We call the elements of $\Omega^2_-(\T^4)$ anti-self-dual 2-forms, and those of $\Omega^2_+(\T^4)$ self-dual 2-forms.
\begin{lemma}\label{lem37}
The family of 2-forms $\{\omega_\bullet\}_{\bullet\in\{I,J,K\}}$ is a basis for $\Omega^2_-(\T^4)$.
\end{lemma}
\begin{proof}
Start by noticing that the degree two cohomology of $\T^4$ has dimension $\binom{4}{2}=6$, in particular: $\dim(\Omega^2_-(\T^4))=\dim(\Omega^2_+(\T^4))=3$. In local coordinates, $\Omega^2(\T^4)$ is generated by an orthonormal basis of the form $\{d_{ij}\}_{1\leq i<j\leq 4}$ (where $d_{ij}:=dx_i\wedge dx_j$). Given that
$$
\begin{array}{ccc}
*d_{12}=d_{34} & *d_{13}=-d_{24} & *d_{14}=d_{23}\\
*d_{23}=d_{14} & *d_{24}=-d_{13} & *d_{34}=d_{12},
\end{array}
$$
\noindent
we can construct an orthogonal basis that preserves the decomposition $\Omega^2(\T^4)=\Omega^2_+(\T^4)\oplus\Omega^2_-(\T^4)$:
$$
\{d_{12}+d_{34},d_{14}+d_{23},d_{13}-d_{24},d_{12}-d_{34},d_{14}-d_{23},d_{13}+d_{24}\},
$$
\noindent
where $\operatorname{Span}\{d_{12}+d_{34},d_{14}+d_{23},d_{13}-d_{24}\}=\Omega^2_+(\T^4)$ and $\operatorname{Span}\{d_{12}-d_{34},d_{14}-d_{23},d_{13}+d_{24}\}=\Omega^2_-(\T^4)$. It remains to note that locally the families $\{\omega_I,\omega_J,\omega_K\}$ and $\{d_{12}-d_{34},d_{14}-d_{23},d_{13}+d_{24}\}$ coincide.
\end{proof}
\noindent
We now have all the necessary tools to state and show the following theorem:
\begin{theorem}\label{thm38}
Using the above notations:
$$
\underline{\mu}^{-1}(\{0\})=\text{Symp}(\T^4\,,\,\omega).
$$
\end{theorem}
\begin{proof}
We have:
\begin{align*}
\underline{\mu}(f)=0\;&\iff\;\forall\;\bullet\in\{I,J,K\},\;\mu_\bullet(f)=0\;\iff\;\forall\;\bullet\in\{I,J,K\},\;f^*\omega_\bullet\wedge\omega=0\\
&\iff\;\forall\;\bullet\in\{I,J,K\},\;\omega_\bullet\wedge(f^{-1})^*\omega=0\;\iff\;((f^{-1})^*\omega)^-=0,
\end{align*}
\noindent
where we used \ref{lem37} for the last equivalence. Now, we have also assumed that $f$ is homotopic to the identity. Therefore, there exists a 1-form $\alpha\in\Omega^1(\T^4)$ such that $f^*\omega-\omega=d\alpha$. Since the LHS form is self-dual, its remains to show that any exact 2-form that is self-dual is zero. Let $d\alpha\in\Omega^2_+(\T^4)$. We have: $* d\alpha=d\alpha$. By definition of the operator $*$ and using Stokes theorem, we get: 
$$
0=\int_{\T^4} d(\alpha\wedge d\alpha)=\int_{\T^4} d\alpha\wedge d\alpha=\int_{\T^4}d\alpha\wedge* d\alpha=\int_{\T^4}g(d\alpha\,,\,d\alpha)vol=G(d\alpha,d\alpha).
$$
\noindent
Which implies that $d\alpha=0$ since $G(.,.)$ is a norm.
\end{proof}
\noindent
For the rest of this section, we will study the application
$$
\Tilde{\mu}(f)=\left((f^{-1})^*\omega\right)^-
$$
\noindent
instead of $\underline{\mu}$ because its differential is simpler and its zero locus is the same (cf \ref{lem37}): $\Tilde{\mu}^{-1}(\{0\})=\SympB$.
\begin{proposition}\label{prop39} $\underline{\mu}$ is $\SympB$-equivariant and $\Tilde{\mu}$ is $\SympB$-invariant.
\end{proposition}
\begin{proof}
Let $\bullet\in\{I,J,K\}$, $f\in\mathcal{M}$, and $\varphi\in\SympB$. We have:
\begin{align*}
&\mu_\bullet(f\circ\varphi)=-\frac{(f\circ\varphi)^*\omega_{I}\wedge\omega}{\text{vol}}=-\frac{\varphi^*f^*\omega_{I}\wedge\varphi^*\omega}{\varphi^*\text{vol}}=\varphi^*\left(-\frac{f^*\omega_{I}\wedge\omega}{\text{vol}}\right)=\mu_\bullet(f)\circ\varphi,\\
&\text{ and:}\\
&\Tilde{\mu}(f\circ\varphi)=\left(((f\circ\varphi)^{-1})^*\omega\right)^-=\left((f^{-1})^*(\varphi^{-1})^*\omega\right)^-=\left((f^{-1})^*\omega\right)^-=\Tilde{\mu}(f).
\end{align*}
\end{proof}

\subsection{DeTurck Trick}
In this subsection, we shall generalize the arguments written in the case of $\T^2$. Recall the two scalar products on $T_f\mathcal{M}$:
\begin{align*}
&G(f_*X\,,\,f_*Y)=\int_{\T^4}(f^*g)(X\,,\,Y)\text{vol}\\
&\Omega_\bullet(f_*X\,,\,f_*Y)=\int_{\T^4}(f^*\omega_\bullet)(X\,,\,Y)\text{vol}.
\end{align*}.
\begin{definition}\label{def310}
Consider the following norm map:
$$
\phi(f)=\frac{1}{2}\|\Tilde{\mu}(f)\|^2=\frac{1}{2}G(\Tilde{\mu}(f),\Tilde{\mu}(f)),
$$
\noindent
and its associated gradient flow that we wish to study:
\begin{equation}\label{eqFGT4}
\partial_t\Tilde{f_t}:=\frac{\partial}{\partial t}\Tilde{f_t}=-\nabla\phi(\Tilde{f_t}).
\end{equation}
\end{definition}
\noindent
We exhibit some properties of this flow:
\begin{proposition}\label{prop311}
\noindent
\begin{enumerate}
    \item $\nabla\phi(f)=-f_*H_f^2i\sharp_gf^*d^*((f^{-1})^*\omega)^-=f_*H_f^2\sharp_\omega f^*d^*((f^{-1})^*\omega)^-$ ;
    \item $\text{Crit}(\phi)=\SympB=\phi^{-1}(0)$ ;
    \item The highest order term of $D_f\nabla\phi(f_*Y)$ is $-f_*H_f^2i\sharp_gf^*d^*d^-\iota_Y(f^{-1})^*\omega$ and the operator is thus not elliptic ;
    \item $\nabla\phi$ is $\SympB$-equivariant.
\end{enumerate}
\end{proposition}
\begin{proof}
\noindent
\begin{enumerate}
\item We first compute $D_f\Tilde{\mu}$:
\begin{align*}
D_f\Tilde{\mu}(f_*Y)&=\partial_s\left((f_s^{-1})^*\omega\right)^-=\left(\partial_s((f_s^{-1})^*\omega)\right)^-\\&=\left(\partial_s((f_s^{-1}\circ f\circ f^{-1})^*\omega)\right)^-=\left(\partial_s((f^{-1})^*(f_s^{-1}\circ f)^*\omega)\right)^-\\&=\left((f^{-1})^*\partial_s(f_s^{-1}\circ f)^*\omega\right)^-=-\left((f^{-1})^*\mathcal{L}_Y\omega\right)^-\\&=-\left((f^{-1})^*(d\iota_Y+\iota_Yd)\omega\right)^-=-d^-(f^{-1})^*\iota_Y\omega,
\end{align*}
\noindent
where we used the definition of the Lie derivative, the fact that if $\partial_s g_s=Y$, then $\partial_s (g_s)^{-1}=-Y$ for all isotopy $\{g_s\}_s$, and Cartan formulae. Then, we compute $D_f\phi$:
\begin{align*}
D_f\phi(f_*Y)&\overset{1}{=}G(D_f\Tilde{\mu}(f_*Y),\Tilde{\mu}(f))=-G(d^-(f^{-1})^*\iota_Y\omega,\Tilde{\mu}(f))\\
&\overset{2}{=}-G(d(f^{-1})^*\iota_Y\omega,\Tilde{\mu}(f))=-G((f^{-1})^*\iota_Y\omega,d^*\Tilde{\mu}(f))\\
&\overset{3}{=}-\int_{\T^4}g((f^{-1})^*\iota_Y\omega,(f^{-1})^*f^*d^*\Tilde{\mu}(f))(f^{-1})^*f^*\text{vol}\\
&\overset{4}{=}-\int_{\T^4}(f^{-1})^*\left(g(\iota_Y\omega,f^*d^*\Tilde{\mu}(f))f^*\text{vol}\right)\\
&\overset{5}{=}-\int_{\T^4}g(\iota_Y\omega,f^*d^*\Tilde{\mu}(f))H_f\text{vol}=-\int_{\T^4}g(\iota_{(YJ)}g,H_ff^*d^*\Tilde{\mu}(f))\text{vol}\\
&\overset{6}{=}-\int_{\T^4}g(YJ,\sharp_gH_ff^*d^*\Tilde{\mu}(f))\text{vol}=\int_{\T^4}g(Y,\left(\sharp_gH_ff^*d^*\Tilde{\mu}(f)\right)J)\text{vol}\\
&\overset{7}{=}\int_{\T^4}f^*\left(g(Y,-\sharp_\omega H_ff^*d^*\Tilde{\mu}(f))\text{vol}\right)=\int_{\T^4}f^*g(Y,-\sharp_\omega H_ff^*d^*\Tilde{\mu}(f))f^*\text{vol}\\
&\overset{8}{=}\int_{\T^4}f^*g(Y,-H_f^2\sharp_\omega f^*d^*\Tilde{\mu}(f))\text{vol},
\end{align*}
\noindent
where we used: at the 2nd line that the two terms are anti-self-dual ; at the 5th line the fact that $f^{-1}$ is an orientation preserving diffeomorphism, the notation $H_f=\frac{f^*\text{vol}}{\text{vol}}$\label{notHF} and the definition of $\omega$ ; at the 6th line the fact that the operators $\flat_g$ and $\sharp_g$ are dual for the metric and the $J$-invariance of $G(\cdot,\cdot)$ ; and finally at the 7th line that $f$ is also an orientation preserving diffeomorphism and the formulae $(\sharp_gX)J=-\sharp_\omega X$. Thus: $\nabla\phi(f)=-f_*(H_f)^2\sharp_\omega f^*d^*((f^{-1})^*\omega)^-$.
\item Let $f\in\text{Crit}(\phi)$. We have $\nabla\phi(f)=0$, which implies that $d^*((f^{-1})^*\omega)^-=0$ since all operators in front of this expression are invertible. On an other hand, since $f\sim Id$, we have: $(f^{-1})^*\omega=\omega+d\alpha$ for a certain 1-form $\alpha$. The condition obtained above then becomes $d(d\alpha)^-=0$. Let us show that this implies that $d\alpha=0$. To do so, we compute:
\begin{align*}
G((d\alpha)^-,(d\alpha)^-)&=\int_{\T^4}(d\alpha)^-\wedge *(d\alpha)^-\text{vol}=-\int_{\T^4}(d\alpha)^-\wedge (d\alpha)^-\\
&=-\int_{\T^4}(d\alpha)\wedge (d\alpha)^-=-\int_{\T^4}\left(d(\alpha\wedge (d\alpha)^-)+\alpha\wedge d(d\alpha)^-\right)\\
&=-\int_{\T^4}d(\alpha\wedge (d\alpha)^-)-\int_{\T^4}\alpha\wedge d(d\alpha)^-=-\int_{\T^4}\alpha\wedge d(d\alpha)^-,
\end{align*}
\noindent
where we used the identity $d(\alpha\wedge (d\alpha)^-)=d\alpha\wedge(d\alpha)^--\alpha\wedge d(d\alpha)^-$ at the 2nd line, and Stokes theorem at the last line. Therefore:
$$
d(d\alpha)^-=0\implies\|(d\alpha)^-\|^2=0\implies (d\alpha)^-=0.
$$
\noindent
It remains to show that $(d\alpha)^-=0\implies d^+\alpha=0$, which will give $(f^{-1})^*\omega=\omega$. But this is clear given the following computation:
\begin{align*}
0=&\int_{\T^4}d(\alpha\wedge d\alpha)=\int_{\T^4}d\alpha\wedge d\alpha\\
&=\int_{\T^4}(d^+\alpha+d^-\alpha)\wedge (d^+\alpha+d^-\alpha)=\int_{\T^4}d^+\alpha\wedge d^+\alpha+\int_{\T^4}d^-\alpha\wedge d^-\alpha\\
&=\int_{\T^4}d^+\alpha\wedge *d^+\alpha-\int_{\T^4}d^-\alpha\wedge *d^-\alpha=G(d^+\alpha,d^+\alpha)-G(d^-\alpha,d^-\alpha),
\end{align*}
\noindent
where we used that the decomposition $d\alpha=d^+\alpha+d^-\alpha$ is $g$-orthogonal, and the identities $d^+\alpha=*d^+\alpha$ and $d^-\alpha=-*d^-\alpha$. On the other hand, we clearly have: $f\in\SympB\implies f\in\operatorname{Crit}(\phi)$ (cf the proof of \ref{prop26}).
\item We notice that the highest order term of the expression:
$$
-\partial_s\left((f_s)_*H_{f_s}^2\sharp_\omega f_s^*d^*((f_s^{-1})^*\omega)^-\right)
$$
\noindent
appears when, after applying successively the Leibniz rule, the differential reaches the term $(f_s^{-1})^*\omega$. This term is of order 2 (because of the $d^*$). The other terms are of order at most 1, hence the result. The operator $D_f\nabla\phi$ isn't elliptic since its kernel contains a space isomorphic to that of closed 1-forms (infinite dimensional).
\item The proof works exactely as in the case of $\T^2$, since $\phi$ is $\SympB$-invariant (cf \ref{prop26}).
\end{enumerate}
\end{proof}
\noindent
We can now mimic the arguments given in the case of $\T^2$ to obtain a parabolic flow from \ref{eqFGT4}: to any diffeomorphism of $\T^4$, we can associate a lift $F:\R^4\rightarrow\R^4$ unique modulo the action by translation of the lattice. Then, since the tangent bundle of the torus is trivial, we can see $F$ as a vector field.
\begin{definition}\label{def312}
We define the operator:
$$
W(f):=-\sharp_{\omega}f^*dd^*\iota_{F^{-1}}\omega,
$$
\noindent
for all $f\in\mathcal{M}$. And we study its associated modified gradient flow:
\begin{equation}\label{FGmT4}
\frac{\partial}{\partial t}f_t=-\nabla\phi(f_t)+(f_t)_*W(f_t)
\end{equation}
\end{definition}
\begin{proposition}\label{prop313}
Let $f\in\mathcal{M}$. Then, $W(f)$ is a Hamiltonian vector field. Moreover, for all $f_*Y\in T_f\mathcal{M}$, we have:
$$
D_fW(f_*Y)=-\sharp_{\omega}\left(\mathcal{L}_Yf^*dd^*\iota_{F^{-1}}\omega+f^*dd^*\iota_{Y}(f^{-1})^*\omega\right).
$$
\end{proposition}
\begin{proof}
We do have:
$$
\iota_{W(f)}\omega=-f^*dd^*\iota_{F^{-1}}\omega=-df^*d^*\iota_{F^{-1}}\omega,
$$
\noindent
which implies that $W(f)$ is the Hamiltonian vector field associated to the Hamiltonian function $f^*d^*\iota_{F^{-1}}\omega$. The computation of the variation is nothing but a direct application of the Leibniz rule:
\begin{align*}
D_fW(f_*Y)&=\left.\frac{\partial}{\partial t}\right|_{t=0}W(f_t)=\left.\frac{\partial}{\partial t}\right|_{t=0}-\sharp_{\omega}f_t^*dd^*\iota_{F_t^{-1}}\omega\\
&=-\sharp_{\omega}\left.\frac{\partial}{\partial t}\right|_{t=0}f_t^*dd^*\iota_{F_0^{-1}}\omega-\sharp_{\omega}f_0^*dd^*\left.\frac{\partial}{\partial t}\right|_{t=0}\iota_{F_t^{-1}}\omega\\
&=-\sharp_{\omega}\left(\mathcal{L}_Yf^*dd^*\iota_{F^{-1}}\omega+f^*dd^*\iota_{Y}(f^{-1})^*\omega\right).
\end{align*}
\end{proof}
\begin{remark}
We note that the term $-\sharp_{\omega}\mathcal{L}_Yf^*dd^*\iota_{F^{-1}}\omega$ is of order 1, whereas the term $-\sharp_{\omega}f^*dd^*\iota_{Y}(f^{-1})^*\omega$ has order 2. On the other hand, we see that, just like the $\T^2$ case, the Hamiltonian vector field $W(f)$ may be absorb by acting on the gradient flow equation by a one-parameter family of Hamiltonian symplectomorphisms (cf \ref{prop210} for the computations).
\end{remark}
\begin{proposition}\label{prop315}
The highest order term of the linearisation of the operator
$$
f\mapsto-\nabla\phi(f)+f_*W(f)
$$
\noindent
at a fixed $f\in\mathcal{M}$ is:
$$
-f_*\sharp_{\omega}f^*\left(H_f^2d^*d^-+dd^*\right)\iota_Y(f^{-1})^*\omega.
$$
\end{proposition}
\begin{proof}
It suffices to take the results of the propositions \ref{prop311} and \ref{prop313}.
\end{proof}
\begin{theorem}\label{thm316}
The order 2 differential operator $L:=H_f^2d^*d^-+dd^*$ acting on $\Omega^1(\T^4)$ is elliptic.
\end{theorem}
\begin{proof}
Let $\alpha\in\Omega^1(\T^4)$. We choose a local chart and write:
$$
\alpha=\sum_{j=1}^4\alpha_jdx_j,
$$
\noindent
where $\{dx_j\}_{j=1}^4$ is a local orthonormal basis of $T^*\T^4$ (since $T^*\T^4\cong\T^4\times\R^4$ we can pick the standard canonical basis of $\R^4$), and the $\alpha_j$ are smooth functions of $\T^4$. We describe the action of $L$ on $\alpha$ step by step.\\
\noindent
\begin{enumerate}
\item We start by describing $d^*(d\alpha)^-$ is those local coordinates. Recall the basis introduced in the lemma \ref{lem37}:
$$
\Omega^2(\T^4)=\operatorname{Span}\{d_{12}+d_{34},d_{14}+d_{23},d_{13}-d_{24},d_{12}-d_{34},d_{14}-d_{23},d_{13}+d_{24}\},
$$
\noindent
where $d_{ij}:=dx_i\wedge dx_j$. It is an orthogonal basis preserving the decomposition $\Omega=\Omega^+\oplus\Omega^-$. And it allows us to describe the projections onto its self-dual and anti-self-dual parts of any $\beta\in\Omega^2(\T^4)$, written locally $\beta=\sum_{1\leq i<j\leq4}\beta_{ij}d_{ij}$:
\begin{align*}
&\beta^+=\frac{1}{2}\left\{(\beta_{12}+\beta_{34})(d_{12}+d_{34})+(\beta_{14}+\beta_{23})(d_{14}+d_{23})+(\beta_{13}-\beta_{24})(d_{13}-d_{24})\right\}\\
&\text{et}\\
&\beta^-=\frac{1}{2}\left\{(\beta_{12}-\beta_{34})(d_{12}-d_{34})+(\beta_{14}-\beta_{23})(d_{14}-d_{23})+(\beta_{13}+\beta_{24})(d_{13}+d_{24})\right\}.
\end{align*}
\noindent
Then, if $\alpha=\alpha_1d_1+\alpha_2d_2+\alpha_3d_3+\alpha_4d_4$ (where $d_i=dx_i$), we successively get:
\noindent
\begin{enumerate}
\item
\begin{align*}
d\alpha&=\left(\partial_1\alpha_2-\partial_2\alpha_1\right)d_{12}+\left(\partial_1\alpha_3-\partial_3\alpha_1\right)d_{13}\\
&+\left(\partial_1\alpha_4-\partial_4\alpha_1\right)d_{14}+\left(\partial_2\alpha_3-\partial_3\alpha_2\right)d_{23}\\
&+\left(\partial_2\alpha_4-\partial_4\alpha_2\right)d_{24}+\left(\partial_3\alpha_4-\partial_4\alpha_3\right)d_{34}\,;
\end{align*}
\item
\begin{align*}
(d\alpha)^-&=\frac{1}{2}\left(\partial_1\alpha_2-\partial_2\alpha_1+\partial_4\alpha_3-\partial_3\alpha_4\right)(d_{12}-d_{34})\\
&+\frac{1}{2}\left(\partial_1\alpha_4-\partial_4\alpha_1+\partial_3\alpha_2-\partial_2\alpha_3\right)(d_{14}-d_{23})\\
&+\frac{1}{2}\left(\partial_1\alpha_3-\partial_3\alpha_1+\partial_2\alpha_4-\partial_4\alpha_2\right)(d_{13}+d_{24})\,;
\end{align*}
\item
\begin{align*}
d^*(d\alpha)^-&=\frac{1}{2}\left((\partial_2^2+\partial_3^2+\partial_4^2)\alpha_1-\partial_2\partial_1\alpha_2-\partial_3\partial_1\alpha_3-\partial_4\partial_1\alpha_4\right)d_1\\
&+\frac{1}{2}\left((\partial_1^2+\partial_3^2+\partial_4^2)\alpha_2-\partial_1\partial_2\alpha_1-\partial_3\partial_2\alpha_3-\partial_4\partial_2\alpha_4\right)d_2\\
&+\frac{1}{2}\left((\partial_1^2+\partial_2^2+\partial_4^2)\alpha_3-\partial_1\partial_3\alpha_1-\partial_2\partial_3\alpha_2-\partial_4\partial_3\alpha_4\right)d_3\\
&+\frac{1}{2}\left((\partial_1^2+\partial_2^2+\partial_3^2)\alpha_4-\partial_1\partial_4\alpha_1-\partial_2\partial_4\alpha_2-\partial_3\partial_4\alpha_3\right)d_4.
\end{align*}
\item Next, we compute in local coordinates $dd^*\alpha$. Since $d^*\alpha$ is nothing but a divergence, we obtain: 
\begin{align*}
dd^*\alpha&=d(\sum_{i=1}^4\partial_i\alpha_i)=\sum_{j=1}^4\sum_{i=1}^4\partial_j\partial_i\alpha_id_j\\
&=\left(\partial_1^2\alpha_1+\partial_1\partial_2\alpha_2+\partial_1\partial_3\alpha_3+\partial_1\partial_4\alpha_4\right)d_1\\
&+\left(\partial_2\partial_1\alpha_1+\partial_2^2\alpha_2+\partial_2\partial_3\alpha_3+\partial_2\partial_4\alpha_4\right)d_2\\
&+\left(\partial_3\partial_1\alpha_1+\partial_3\partial_2\alpha_2+\partial_3^2\alpha_3+\partial_3\partial_4\alpha_4\right)d_3\\
&+\left(\partial_4\partial_1\alpha_1+\partial_4\partial_2\alpha_2+\partial_4\partial_3\alpha_3+\partial_4^2\alpha_4\right)d_4.
\end{align*}
\end{enumerate}
\item We can now extract the principal symbol $\sigma_L$ of the operator $L$ out of the two previous local computations. Let $\xi=(x,y,z,t)\in T^*_p\T^4\cong\R^4$ and $h:=H_f$. We have:
\begin{small}
\begin{align*}
&\sigma_L(\xi)=\left(\begin{array}{cccc}
x^2+\frac{h}{2}(y^2+z^2+t^2)&(1-\frac{h}{2})xy&(1-\frac{h}{2})xz&(1-\frac{h}{2})xt\\
(1-\frac{h}{2})xy&y^2+\frac{h}{2}(x^2+z^2+t^2)&(1-\frac{h}{2})yz&(1-\frac{h}{2})yt\\
(1-\frac{h}{2})xz&(1-\frac{h}{2})yz&z^2+\frac{h}{2}(x^2+y^2+t^2)&(1-\frac{h}{2})zt\\
(1-\frac{h}{2})xt&(1-\frac{h}{2})yt&(1-\frac{h}{2})zt&t^2+\frac{h}{2}(x^2+y^2+z^2)
\end{array}\right)
\end{align*}
\end{small}
\noindent
And we check that this symbol in invertible for each $\xi\neq0$:
\begin{align*}
\operatorname{det}(\sigma_L(\xi))&=\left(\frac{h}{2}\right)^3(t^8+x^8+4t^2x^6+6t^4x^4+4t^6x^2+y^8+4t^2y^6+4x^2y^6+6t^4y^4+6x^4y^4\\
&+12t^2x^2y^4+4t^6y^2+4x^6y^2+12t^2x^4y^2+12t^4x^2y^2+z^8+4t^2z^6+4x^2z^6\\
&+4y^2z^6+6t^4z^4+6x^4z^4+12t^2x^2z^4+6y^4z^4+12t^2y^2z^4+12x^2y^2z^4\\
&+4t^6z^2+4x^6z^2+12t^2x^4z^2+12t^4x^2z^2+4y^6z^2+12t^2y^4z^2\\
&+12x^2y^4z^2+12t^4y^2z^2+12x^4y^2z^2+24t^2x^2y^2z^2)\\
&=\left(\frac{h}{2}\right)^3(x^2+y^2+z^2+t^2)^4,
\end{align*}
\noindent
since $f\sim\operatorname{Id}$, it follows that $h=H_f\neq0$. Thus, $\operatorname{det}(\sigma_L(\xi))\neq0$ for all $\xi\neq0$. The operator $L=H_f^2d^*d^-+dd^*$ is elliptic.
\end{enumerate}
\end{proof}

\subsection{Theorem for $\T^4$}\label{thmB}
\begin{theorem}
Let $k\in\mathbb{N}$ and $\alpha\in(0,1)$. For all $f\in\mathcal{M}^{k,\alpha}$, there exists $\epsilon(f)>0$ and a unique one-parameter family $\{\Tilde{f_t}\}$ of $\mathcal{M}^{k,\alpha}$ satisfying $\Tilde{f_0}=f$ and that solves the problem \ref{eqFGT4} for all $t\in[0,\epsilon)$. Moreover, any $\Tilde{f_t}$ is smooth.
\end{theorem}
\begin{proof}
The proof is decomposed into 4 pieces.\\
\noindent
\begin{enumerate}
\item The equation \eqref{FGmT4} is parabolic by proposition \ref{thm316} so the result is certainly true for this equation with initial condition $f_0=f\in\mathcal{M}^{k,\alpha}$: a unique solution $\{f_t\}$ of \eqref{FGmT4} exists on a time interval $[0,\epsilon(f))$ (usual local existence theorem for strongly parabolic flow cf \cite{AL00} chapter 7 and 8).
\item We now solve the ODE associated to the flow of $W(f_t)$ point-wise (cf \ref{def312}). Since $\T^4$ is compact, the solution admits a unique solution that exists for as long as $\{f_t\}$ exists (cf \ref{prop213}). Moreover, it is as regular as $\{f_t\}$.
\item Hence, by letting $\{\Tilde{f_t}:=f_t\circ\varphi_t\}$, we obtain a solution of \eqref{eqFGT4}. This solution exists for all $t\in[0,\epsilon(f))$ and must be as regular as the least regular of the two applications of the composition (hence as regular as $\{f_t\}$).
\item We end the discussion with the regularity. The elliptic regularity for compact manifolds implies that the family $\{f_t\}$ is smooth by bootstrapping, and since the flow of $W(f_t)$ is as regular as the family $\{f_t\}$ (smooth), it follows that $\{\Tilde{f_t}\}$ is smooth as a composition of smooth functions.
\end{enumerate}
\end{proof}
\noindent
More details about the local existence theorem can be found in my thesis \cite{LPMT}.

\section{Conclusion and perspectives}
The short time existence of the gradient flow shows the local contractibility of the group of symplectomorphisms of the two and four dimensional tori. But this was expected. The natural question that should follow is whether this flow exists for all time or not. And if not, can we use some analytical tools to capture the obstruction. The work to be accomplished is to develop a continuity method adapted to the infinite dimensional setting (more specifically the Banach manifold setting) in order to test the long time existence of the flow, or to develop a perturbation theory for those equations and study their moduli in a more algebraic geometric approach.


\begin{thebibliography}{99}

\bibitem{MS16}D. McDuff, and D. Salamon: 
\newblock \emph{Introduction to symplectic topology.}
\newblock Oxf. Grad. Texts Math. \textbf{27} (2016), Oxford University Press.

\bibitem{AM00}M. Abreu, and D. McDuff: 
\newblock \emph{Topology of symplectomorphism groups of rational ruled surfaces.}
\newblock J. Am. Math. Soc. \textbf{13} (2000).

\bibitem{YR24}Y. Rollin: 
\newblock \emph{Symplectic maps and hyperKähler moment map geometry.}
\newblock (2024) \texttt{arXiv:2110.02679.}

\bibitem{AT22}Alice B. Tumpach: 
\newblock \emph{Structures géométriques sur les variétés banachiques.}
\newblock (2022) \texttt{\url{https://geometricgreenlearning.com/research_files/Cimpa.pdf}.}

\bibitem{CK04}B. Chow, and D. Knopf: 
\newblock \emph{The Ricci Flow: An Introduction.}
\newblock Mathematical Surveys and Monographs \textbf{110} (2004), American Mathematical Society.

\bibitem{AL00}A. Lunardi: 
\newblock \emph{Analytic Semigroups and Optimal Regularity in Parabolic Problems.}
\newblock Modern Birkhäuser Classics (1995), Birkhäuser Basel.

\bibitem{LPMT}L. Pinsard Morel: 
\newblock \emph{\'Etude du flot gradient d'application moment sur les tores de dimension 2 et 4.}
\newblock (soon)
\texttt{\url{https://nantes-universite.hal.science/LMJL-THESE}.}

\end{thebibliography}
\end{document}